\documentclass[letterpaper,10pt,reqno]
{amsart}
\usepackage{amsmath}
\usepackage{amssymb}
\usepackage{times}
\usepackage{hyperref}
\usepackage{graphicx}
\usepackage{color}
\usepackage[small]{caption2}
\usepackage{booktabs}
\usepackage[sort,numbers]{natbib}
\usepackage{paralist}

\newcommand{\ul}{\underline}
\newcommand{\ol}{\overline}

\hypersetup{
    bookmarks=true,         
    unicode=false,          
    pdftoolbar=true,        
    pdfmenubar=true,        
    pdffitwindow=true,      
    pdftitle={My title},    
    pdfauthor={Author},     
    pdfsubject={Subject},   
    pdfnewwindow=true,      
    pdfkeywords={keywords}, 
    colorlinks=true,       
    linkcolor=blue,          
    citecolor=blue,        
    filecolor=blue,      
    urlcolor=blue           
}

\newcommand{\inprob}	{\overset{\mathcal{P}}{\to}}

\newcommand{\ex}[1]		{\mathbb{E}\left[ #1 \right]}
\newcommand{\exc}[1]	{\mathbb{E}[ #1 ]}
\newcommand{\pr}[1]		{\mathbb{P}\left\{ #1 \right\}}
\newcommand{\prc}[1]	{\mathbb{P}\{ #1 \}}
\newcommand{\prm}[1]	{\mathbb{P}\big\{ #1 \big\}}

\newcommand{\var}[1]	{\mathbf{Var}\left( #1 \right)}
\newcommand{\event}[1]	{\left\{ #1 \right\}}

\renewcommand{\exp}[1]	{\operatorname{exp}\left( #1 \right)}

\newcommand{\floor}[1]	{\left\lfloor #1 \right\rfloor}
\newcommand{\floorc}[1]	{\lfloor #1 \rfloor}
\newcommand{\ceil}[1]	{\left\lceil #1 \right\rceil}

\newcommand{\eps}		{\varepsilon}
\newcommand{\emptystring} {\emptyset}
\newcommand{\R}		{\mathbb{R}}

\newcommand{\cB}		{\mathcal{B}}
\newcommand{\cV}		{\mathcal{V}}

\DeclareMathOperator*{\argmax}{arg\,max}

\newcommand{\cA}{\mathcal{A}}

\newtheorem{theorem}{Theorem}[section]
\newtheorem{lemma}[theorem]{Lemma}
\newtheorem{proposition}[theorem]{Proposition}

\theoremstyle{definition}

\makeatletter
\def\timenow{\@tempcnta\time
  \@tempcntb\@tempcnta
  \divide\@tempcntb60
  \ifnum10>\@tempcntb0\fi\number\@tempcntb
  \multiply\@tempcntb60
  \advance\@tempcnta-\@tempcntb:\ifnum10>\@tempcnta0\fi\number\@tempcnta}
\makeatother

\title{Longest path distance in random circuits}
\date{\today}                                           

\author{Nicolas Broutin}
\address{INRIA Rocquencourt, 78153 Le Chesnay, France}
\email{nicolas.broutin@inria.fr}

\author{Omar Fawzi}
\address{School of Computer Science, McGill University,  H3A 2K6, Montreal, Canada}
\email{ofawzi@cs.mcgill.ca}

\begin{document}
\maketitle

\begin{abstract}
We study distance properties of a general class of random directed acyclic graphs (\textsc{dag}s). In a \textsc{dag}, many natural notions of distance are possible, for there exists multiple paths between pairs of nodes. The distance of interest for circuits is the maximum length of a path between two nodes. We give laws of large numbers for the typical depth (distance to the root) and the minimum depth in a random \textsc{dag}. This completes the study of natural distances in random \textsc{dag}s initiated (in the uniform case) by Devroye and Janson (2009+). We also obtain large deviation bounds for the minimum of a branching random walk with constant branching, which can be seen as a simplified version of our main result.
\end{abstract}

\section{\bf Introduction}

Motivated by the circuit value problem, and the delay to evaluate the output, \citet*{DSSTT94} initiated the study of depths in random circuits. The model is that of uniform random circuits \cite{DL95}: a random circuit is built by iterative addition of gates, each gate randomly choosing $k$ inputs among the outputs of the gates already present. The model has been further studied by \citet*{AGM99} and \citet{TX96}. Writing $D_x$ for the depth of the gate arrived at step $x$, \citet{TX96} proved that the depth a random circuit of $n$ gates, $\max\{D_x : x\le n\}$ is asymptotic to $ke \log n$ in probability ($\log$ denotes the natural logarithm).

Actually, in a random circuit, many distinct directed paths may link two gates, and one can define different notions of distances. 
\citet{DJ09} started the systematic analysis of these distances. Among those, the distances defined by paths that are only allowed to look one step ahead (``greedy'' distances) were studied by \citet*{Mah09} and \citet*{DFF10}. These greedy distances have permitted to quantify the effect of the ``power of choice'' for depths in random trees: like a gate, each node has $k$ potential contacts, but only attaches to the most desirable according to a measure of optimality. The models in which the choice of each node is made only according the labels of its potential contacts has been studied by \citet*{Mah09} and \citet*{DFF10}. \citet*{DKM07} studied more general rules of growing trees, where the choice of each node might depend on the degree or the distance to the root of the potential ancestors.

Unfortunately, the distance that has the most meaning in terms of performance of circuits cannot be defined in a greedy way: the number of layers of a circuit depends on the \emph{maximum} length of a path between an output and an input. In this paper, we study precisely this distance, and hence the number of layers required to evaluate the random circuit. Aside from the depth of the entire circuit studied in \citet{AGM99} and \citet{TX96}, the typical depth $D_n$ and the minimum depth of a gate are also of interest. For the latter quantity, since $\min\{D_x: 0\le x\le n\}=0$, we study $\min\{D_x: x\ge n/2\}$ to estimate the concentration of the depths in the circuits. Our main results are laws of large numbers for $D_n$ (Theorem~\ref{thm:typlongest}) and $\min\{D_x: x\ge n/2\}$ (Theorem~\ref{thm:minlongest}). In particular, for the model of uniform random circuits of \cite{TX96} and $k=2$, we show that
\begin{equation}\label{eq:lln}
\frac{D_n}{\log n} \to\lambda \qquad \text{ and } \qquad \frac{\min_{n/2 \leq x \leq n} D_x}{\log n} \to \lambda_{\min}=\frac{\lambda}{2}
\end{equation}
in probability, where $\lambda = 4.31107\dots$ is the only solution to the equation $\lambda \log (2e/\lambda)=1$ that is greater than one. Doing so, we prove a conjecture of \citet{DJ09} about the value of $\lambda$, and we also identify $\lambda_{\min}$ for which they did not have a guess. This completes the study of some natural distances in \emph{uniform} random circuits started in \cite{TX96} and \cite{DJ09}. In fact, our results apply to a more general class of random \textsc{dag}s where the parent nodes of $x$ are not necessarily chosen uniformly from $\{0, \dots, x-1\}$ \cite{DFF10} (we will be more precise shortly). In general, the limit constants are characterized uniquely as the root of some (often implicit) equation that depends on the precise model of attachment.

The problem of distances in random \textsc{dag}s is related to minima in branching random walks. The relation between distances in \emph{random trees} such as random recursive trees, or binary search trees and minima in branching random walks has been exhibited by \citet{Pittel1985} and \citet{Devroye1986}. Distances in \emph{random circuits} can also be studied using the simpler setting of branching random walk, but the relation is much more intricate because the circuits do not have a real tree structure.
Although we do not use the results about branching random walks directly, we think that the reader may find it useful to warm up with this simpler model. Moreover, the ideas leading to the tail bounds for minima in branching random walks presented in Theorems~\ref{thm:BRW-righttail} 
%
%
also underly the main argument behind our analysis of the behaviour of $\min_{n/2 \leq x \leq n} D_x$.  

\medskip
	\noindent\textsc{Further bibliographic remarks.} For a slightly different random circuit model, Mahmoud and Tsukiji have investigated the asymptotic behaviour of the number of outputs, that is gates that do not feed in any other gate \cite{TsMa2001a,MaTs2004a}. The profile of the related model of $k$-trees has been studied by \citet*{DaHwBoSo2010a}. Depths in random circuits are also used by \citet{CoGeSi1995a} in relation to parallel computations time.

\medskip
\noindent\textsc{Outline of the paper.} The model of random \textsc{dag} is formally introduced in Section~\ref{sec:def}. In Section~\ref{sec:BRW}, we study tail bounds for minima of branching random walks. The results presented there help understand why the values for the limiting constants in the law of large numbers in (\ref{eq:lln}) are what they are.  To the best of our knowledge, the exponential rates in the tail estimates we derive for the branching random walk were not known before. Section~\ref{sec:typical_distance} is then devoted to the study of the typical number of edges $D_n$ on the longest path between $n$ and the root of a random \textsc{dag}. Finally, the minimal distance between a node (with sufficiently large label) and the root is analyzed in Section~\ref{sec:minimum_distance}. Although we do not think that the results for branching random walks are obvious, the main difficulty consists, in the \textsc{dag} model, in dealing with the intricate dependence between the different paths up the root that originate not only from a single node, but also from different nodes.

\section{\bf Definitions and notation}\label{sec:def}



We consider the more general model of scaled attachment random recursive \textsc{dag}s ($k$-\textsc{sarrd}) introduced in \cite{DFF10}. We are given a random variable $X$, with support in $[0,1)$. In a $k$-\textsc{sarrd} with attachment $X$ (or $(X,k)$-\textsc{sarrd}), every node $x$ chooses $k$ parents: $\floor{x X_{x,1}}, \floor{x X_{x,2}}, \dots, \floor{x X_{x,k}}$ where  $X_{0,1}, \dots, X_{0, k}$, \dots, $X_{n,1}, \dots, X_{n,k}$ are independent copies of $X$. In other words, the random variable $X_{x,p}$ determines the $p$-th parent of node $x$. A random \textsc{dag} of size $n$ is then composed of the root $0$ and the nodes $1,\dots, n$ and the edge set binding each node to its $k$ parents. When $X$ is uniform, one obtains the uniform recursive circuit ($k$-\textsc{urrd}) that is the subject of \cite{TX96, DJ09, AGM99}.\footnote{One can also consider close variants of the model where the parents are chosen without replacement and the \textsc{dag} has $k$ roots \cite{DL95}.}

In the following, we always reserve $k$ for the number of parents in the \textsc{dag}. 
We let $D_n$ represent the length of the longest path from node $n$ to node $0$.  

We now introduce some notation to describe the \textsc{dag}. The set of finite words on the alphabet $\mathcal A=[k]=\{1,\dots, k\}$ is denoted by
$$\mathcal U:=\bigcup_{m\ge 0} \mathcal A^m.$$ The set $\mathcal U$ is naturally endowed with a partial order: we write $v\preceq u$ if $v$ is a prefix of $u$. We will also think of $\mathcal U$ as a $k$-ary tree, where $v\preceq u$ if $v$ is an ancestor of $u$.

For a node $x$ and a string $s \in \mathcal U$, $L(x,s)$ is the label of the ancestor of $x$ obtained by following the path labeled by $s$. For example, $L(x,1)$ is the first parent of node $x$. 
Concatenation of strings (that correspond to path) is denoted by $\cdot$\,. Note that in our model we have $L(x,s\cdot p) = \lfloor L(x,s)X_{L(x,s), p}\rfloor$. 
For a string $s$ and $i\ge 0$, let $\underline s_{i}$ be the string composed of the first $i$ letters of $s$. In the special case where $i=|s|-1$, we write $ s^-:=\ul s_{|s|-1}$ for the string where the last letter is dropped. The last letter of the word $s$ is denoted by $\ol s$ so that $s=s^- \cdot \ol s$. 

\begin{figure}[htb]
\small
\begin{picture}(100,130)
\put(0,0){\includegraphics[scale=.5]{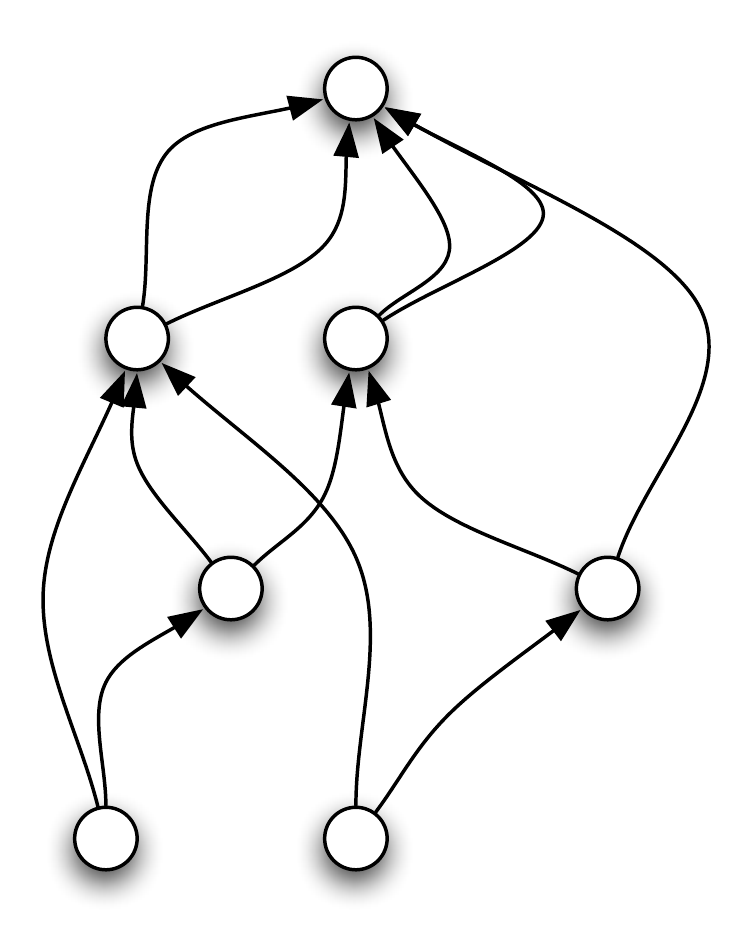}}
\put(49,122){0}
\put(18,86){1}
\put(49,86){2}
\put(86,50){3}
\put(49,14){4}
\put(31,50){5}
\put(13,14){6}
\end{picture}
\begin{picture}(150,130)
\put(14,0){\includegraphics[scale=.5]{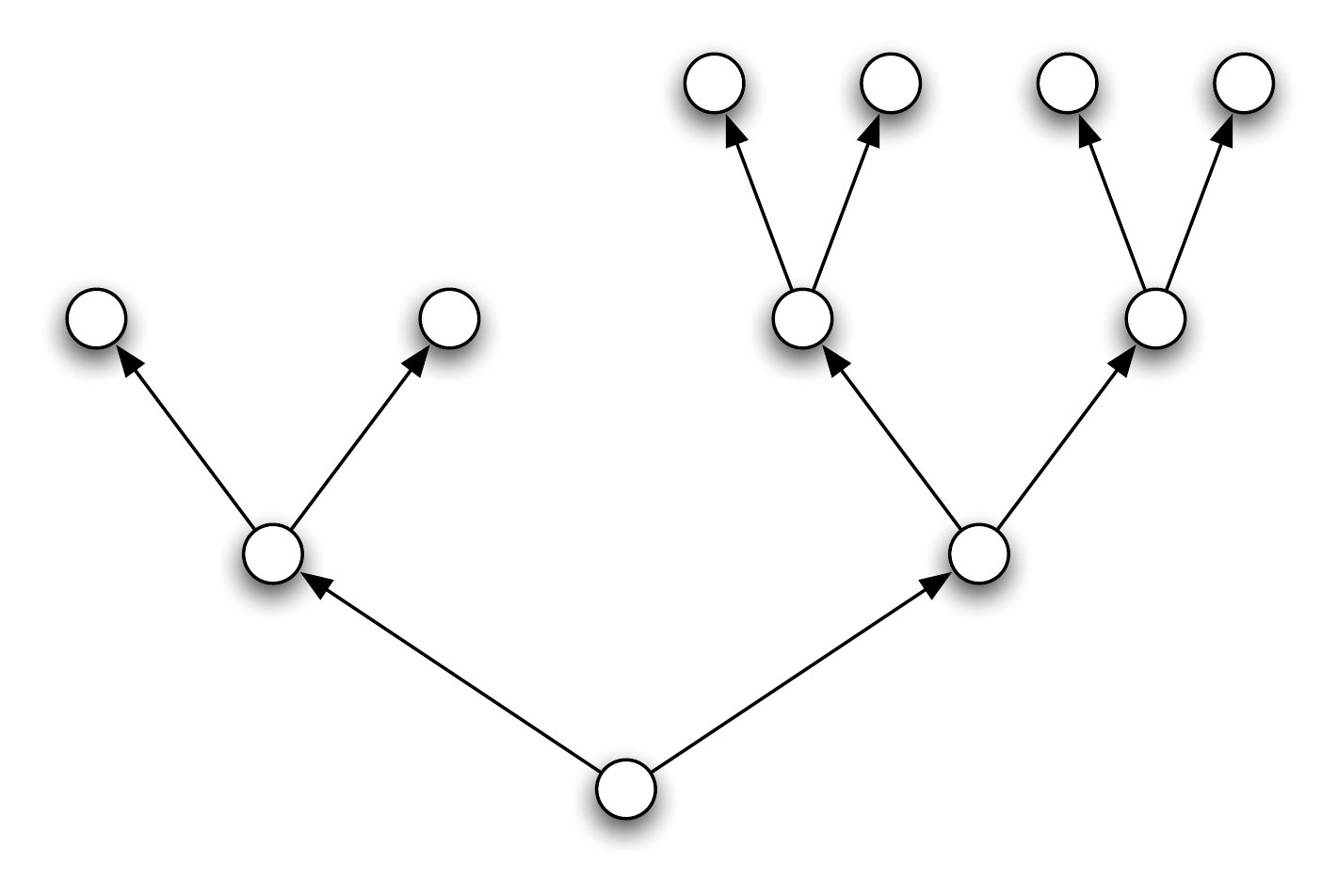}}
\put(108,14){6}
\put(162,50){5}
\put(53,50){1}
\put(26,86){0}
\put(80,86){0}
\put(135,86){1}
\put(189,86){2}
\put(121,122){0}
\put(148,122){0}
\put(175,122){0}
\put(203,122){0}
\end{picture}
\caption{\label{fig:dag}A binary dag on $\{0,1,\dots, 6\}$ and the corresponding exploration tree from node $6$.}
\end{figure}
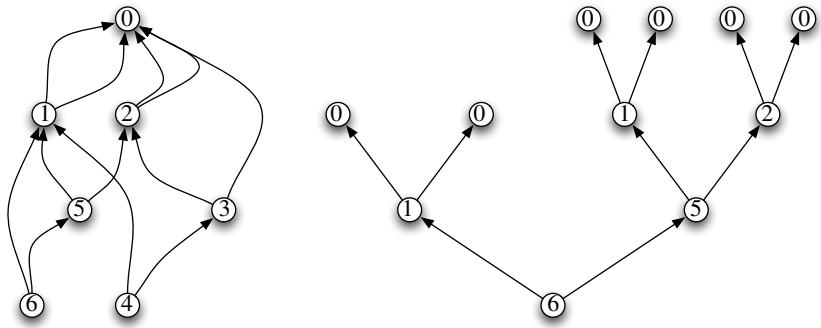

Of course, the paths up the root corresponding two strings $s$ and $s'$ might intersect, so that even when $s$ and $s'$ have no non-trivial common prefix, the random variables $L(n,s)$ and $L(n,s')$ are not independent in general (see Figure~\ref{fig:dag}). However, the dependence of random variables $L(n,s)$, $s\in \mathcal U$ is ``essentially caused'' by common prefixes of the strings $s$. We will later justify precisely this informal fact as we prove our main results in Sections~\ref{sec:typical_distance} and~\ref{sec:minimum_distance}. Understanding the lengths of the paths originating from a node $n$ hence somewhat reduces to the analysis of the evolutions of the labels $L(n,s)$ in the $k$-ary \emph{tree} $\mathcal U$.
Furthermore, the reader should be intuitively convinced that the labels should satisfy the approximation
$$L(n,s)\approx n\prod_{u\preceq s}X_u,$$
where $X_u, u\in \mathcal U$ are i.i.d.\ copies of $X$, which makes the connection between paths in $X$-\textsc{dag}s and a branching random walk with step distribution $Y = -\log X$. Since we are interested in \emph{maximum} lengths of paths, we should naturally study paths along which the $L(x,s)$ stays large; along these paths, the branching random walk should be small. This leads us to the study of asymptotics for \emph{minima} in a branching random walk.

\section{\bf Large deviations for extremes of a branching random walk}
\label{sec:BRW}

In this section, we consider branching random walks with constant branching factor $k$ and step distribution $Y$.   Let $Y_u$, $u\in \mathcal U\setminus \varnothing$ be independent and identically distributed random variables distributed as $Y$. It is convenient to assume that $Y_\varnothing=0$. Then, define the \emph{position} of a word $u\in \mathcal U$ by $S_u=\sum_{v\preceq u} Y_v$. We are interested in the minimum label over the $k^m$ nodes at distance $m$ from the root, $M_m=\min\{ S_u: u\in \cA^m\}$. 

Asymptotics for $M_m$ depend on rate functions associated with the step distribution $Y$. Let $\Lambda$ be the cumulant generating function for the step distribution $Y$ and define its convex dual $\Lambda^\star$ \cite{DZ98, Rockafellar1970} by
\[
\Lambda(\lambda) = \log \exc{e^{\lambda Y}} \qquad \text{ and } \qquad \Lambda^\star(z) = \sup_{\lambda\in \R} \{ \lambda z - \Lambda(\lambda)\}
\]
where $\log$ denotes the natural logarithm. The following classical theorem describes the first order asymptotics of $M_m$. In the entire document, we will always have $Y\ge 0$, so that the moment condition $\Lambda(\lambda)<\infty$ for some $\lambda<0$ will always be satisfied.


\begin{theorem}[\cite{Ham74, Kin75, Big76}]
\label{thm:minbrw1}
Suppose $Y$ is such that $\Lambda(\lambda) < \infty$ for some $\lambda<0$. Let $\gamma = \inf \{z \leq \ex{Y} : \Lambda^\star(z) < \log k\}$. Then, $\gamma$ is finite and with probability $1$,
\[
\lim_{m \to \infty} \frac{M_m}{m} = \gamma.
\]
Moreover, if $\ex{Y} < \infty$, by the dominated convergence theorem, we have
\[
\lim_{m \to \infty} \frac{\ex{M_m}}{m} = \gamma.
\]
\end{theorem}

Results on the minima in branching random walks have many applications in the study of random trees. See \citet{Dev98} for a survey. Here, we are interested in tail bounds for the distribution of $M_m/m$. \citet{McD95} and \citet{AR09} have proved
%
%
%
general exponential tail bounds for the deviations of $M_m/m$; however, the exponential rates in the bounds there are not optimal, and our setting requires to identify them. In the following, we assume that $Y$ is asymptotically exponential in the following sense: there exists a constant $\alpha\in(0,\infty)$ such that
\begin{equation}\label{eq:exp_tails}\frac{\log \pr{Y\ge x}}x \to -\alpha,\end{equation}
as $x\to\infty$. 


The two sides of the distribution of $M_m/m$ have in general very different behaviour. Quite intuitively, if one wants to make the minimum value at level $m$ smaller, it suffices to modify the random variables $Y_u$ on a single path of length $m$, so one expects that the tail should have exponential tails with a scale of $m$ on the left. On the other hand, to make $M_m$ larger, one needs to modify \emph{all} $k^m$ paths of length $m$, and it is not immediately clear how one should do this in order to optimize the probability. We will show that when the random variable $Y$ has exponential tails in the sense of (\ref{eq:exp_tails}), it is essentially best to modify the random variables on the first level of the tree by a huge amount. We now turn to formalizing this intuition.

\begin{theorem}[Right tail] 
\label{thm:BRW-righttail}
Suppose that $Y \geq 0$, $\log \prc{Y\ge x}= -\alpha x+o(x)$, as $x\to\infty$, for $\alpha\in (0,\infty)$.
Let $\eps, \delta > 0$. Then there exists constants $c,c'\in (0,\infty)$ such that, for all $m$ large enough,
\[
c \cdot e^{-k \alpha\eps(1 + \delta) m} \leq \pr{M_m \geq (\gamma + \eps) m} \leq c' \cdot e^{-k \alpha\eps(1-\delta) m}.
\]
\end{theorem}
\begin{proof}
We start with the lower bound. Consider the $k$ nodes connected to the root. Each one of these nodes is the root of a tree of depth $m-1$. In order to have $M_m\ge (\gamma+\eps) m$, it suffices that all $k$ (independent) trees have a minimum $M_{m-1}\ge (\gamma-\frac{\eps \delta}{2})m$ and all the steps between the root and its children are such that $Y_{1}, \dots, Y_k \geq (\eps+\frac{\eps \delta}{2})m$.
Therefore,
\begin{align*}
 \pr{M_m \geq (\gamma + \eps) m} 
&\geq \pr{M_{m-1} \geq (\gamma - \eps \delta/2) m}^k \cdot \pr{Y \geq (\eps + \eps \delta/2) m}^k \\
&\geq c \cdot e^{-k \alpha (\eps + \frac{\eps\delta}{2}) m+o(m)}\\
&\geq c \cdot e^{-k \alpha \eps(1 + \delta)m},
\end{align*}
for $m$ large enough,
by Theorem \ref{thm:minbrw1} and our assumption on the tail of $Y$.

We now prove the upper bound.
%
%
%
%
Let $h = \floor{10 \log_k m}$. We start by proving an exponential tail bound for $S_u$ where $u \in \cA^h$. Notice first that  
$$\exc{e^{\alpha (1-\delta) Y}} = \int_{0}^{\infty} \prc{e^{\alpha (1-\delta)Y} \geq t} dt = \int_0^{\infty} t^{-1/(1-\delta)+ o(1)} dt < \infty.$$
Thus, by Markov's inequality, we obtain
\[
\pr{S_u \geq (1-\delta) \eps m} \leq \exc{e^{\alpha (1-\delta)Y}}^{h} e^{-\alpha(1-\delta) \cdot (1-\delta)\eps m} \leq e^{-(1-\delta)^3\alpha \eps m},
\]
for $m$ large enough. We can now prove exponential tails for $M_m$. Using a decomposition according to the values of some nodes at level $h+1$, we obtain
\begin{align}\label{eq:decomp_mn}
\pr{M_m \geq (\gamma+\eps)m} 
&\leq \prc{\forall p \in [k], \exists u \in \cA^{h} : S_{p \cdot u} > (1-\delta)\eps m} \\
&\quad+\prc{\exists p \in [k], \forall u \in \cA^{h}: S_{p\cdot u}\le (1-\delta)\eps m \text{~and~} M_m\ge (\gamma+\eps)m}.\nonumber
\end{align}
For the first event in \eqref{eq:decomp_mn} to hold, there must be $k$ paths $u^1,\dots, u^k$ such that, for the $k$ \emph{disjoint} paths $p\cdot u^k$ at level $h+1$, the values are rather large:
\begin{align}\label{eq:decomp_mn_bad}
\prc{\forall p \in [k], \exists u \in \cA^{h} : S_{p \cdot u} > (1-\delta)\eps m}
&\le \pr{\forall p \in [k], \exists u \in \cA^{h} : S_{p \cdot u} > (1-\delta)\eps m} \nonumber\\
&\le k \cdot k^{kh} e^{-k (1-\delta)^3 \alpha \eps m}.
\end{align}
On the other hand,
\begin{align*}
&\prc{\exists p \in [k], \forall u \in \cA^{h}: S_{p\cdot u}\le (1-\delta)\eps m \text{~and~} M_m\ge (\gamma+\eps)m}\\
&= \prm{\exists p \in [k], \forall u \in \cA^{h} : \left[ S_{p \cdot u} \leq (1-\delta)\eps m \; \text{ and } \; \forall v \in \cA^{m-h-1} : S_{p \cdot u \cdot v} \geq (\gamma+\eps)m \right] } \\
&\le \pr{M_{m-h-1} \geq (\gamma+\delta \eps)m}^{k^h} \\
&\le c_0^{k^h}
\end{align*}
for $m$ large enough and some constant $c_0< 1$ (using Theorem \ref{thm:minbrw1}). As a consequence, for fixed $\delta,\epsilon>0$, we have for all $m$ large enough
\begin{align*}
\pr{M_m\ge (\gamma+\epsilon)m} 
&\le k^{kh+1} e^{-k (1-\delta)^3 \alpha \eps m} +c_0^{k^h}.
\end{align*}
By our choice for $h=\floor{10\log_k m}$ we have $k^h\ge m^{10}/k$. Finally, since $\delta>0$ was arbitrary, 
 the desired upper bound follows.\end{proof}

At this point, we should comment on tail bounds for minima in branching random walks. Although the estimates for the upper tail of $M_m$ in Theorem~\ref{thm:BRW-righttail} are tight, we must mention that the strength of the result does not really compare to recent results on the precise location of minima in branching random walk. Indeed, Theorem~\ref{thm:BRW-righttail} only provides decent estimates for $\prc{M_m\ge \gamma m +t}$ for \emph{positive} $t$, while $\exc{M_m}=\gamma m - \beta \log m +O(1)$ \cite{AR09}; for the study of branching random walks the tail bounds of interest are actually those for $\prc{M_m \ge \exc{M_m}+ t}$ or $\prc{M_m \ge \gamma m - \beta \log m +t}$. See also the related results about tightness and weak convergence for $M_m$ \cite{Bramson1978a,HuSh2009,Bachmann2000,Aidekon2011,BrZe2009}. Equivalent results for the height of random trees (binary search trees and $m$-ary search trees) were proved in \cite{Reed2003,Drmota2003,ChDr2006}. Similar comments apply for our estimates on the left tail that follow.

\medskip
The left tail for the minimum at level $m$ in a branching random walk is essentially governed by the level $m$: To change the minimum $M_m$, it suffices to change one single of the $k^m$ paths of length $m$ and this way of proceeding is essentially optimal.

\begin{theorem}[Left tail]
\label{thm:BRW-lefttail}
Suppose that there exists $\lambda<0$ such that $\exc{e^{\lambda Y}}<\infty$. Then, for any $\delta>0$, there exists constants $c,c'$ such that 
\[
ck^m e^{-m (\Lambda^\star(\gamma-\eps)+\delta)} \leq \pr{M_m \leq (\gamma - \eps) m} \leq c'k^m e^{- m\Lambda^\star(\gamma-\eps) }.
\]
\end{theorem}
\begin{proof}
The upper bound follows easily from the union and Chernoff's bound \cite{Che52,DZ98}:
\begin{align*}
\pr{M_m \le (\gamma -\eps)m } 
&\le k^m \pr{\sum_{i=1}^m Y_i \le (\gamma-\eps)n}\\
&\le k^m e^{-n\Lambda^\star(\gamma-\eps)}.
\end{align*}

The lower bound is proved using a branching process argument. Let $L\ge 1$ be an arbitrary integer to be chosen later. The potential individuals of our branching process are the nodes of $\mathcal U$ at levels $iL$, $i\ge 0$. A node $u$ is called \emph{good} if either it is the root (lies at level $0$), or it lies at level $(i+1)L$ for some $i\ge 0$, its ancestor $v$ at level $iL$ is good and $S_u-S_v\le (\gamma-\eps)L$. Let $Z_i$ denote the number of good nodes at level $iL$ in the tree ; $\{Z_i, i\ge 0\}$ is a Galton--Watson process. Clearly, if there is a good node at level $iL$, i.e., $Z_i>0$ then $M_{iL}\le (\gamma-\eps) iL$. As a consequence,
\begin{equation}\label{eq:GW_embedding}
\pr{M_{iL} \le (\gamma-\eps){iL}}
\ge \pr{Z_i>0}.
\end{equation}
By the second moment method \cite{AlSp2008}, more precisely the Chung--Erd\H{o}s inequality \cite{CE52}, we have
\begin{equation}\label{eq:chung_erdos}
\pr{Z_i>0}\ge \frac{\ex{Z_i}^2}{\ex{Z_i^2}}=\frac{\ex{Z_i}^2}{\var{Z_i}+\ex{Z_i}^2}.	
\end{equation} 
For the Galton--Watson process $\{Z_i, i\ge 0\}$, one has
$$\ex{Z_i}=\mu^i\qquad \text{and}\qquad \var{Z_i}=\sigma^2 \frac{1-\mu^i}{1-\mu}\mu^{i-1},$$
where $\mu:=\ex{Z_1}$ and $\sigma^2:=\var{Z_1}$ \cite{AtNe1972}. 
We now move on to choosing $L$ to obtain a good lower bound on the right-hand side in (\ref{eq:chung_erdos}). The mean number of children of an individual is
$$
\mu=\ex{Z_1}= k^L \pr{\sum_{j=1}^L Y_j\le (\gamma-\eps)L} = k^L e^{-L \Lambda^\star(\gamma-\eps) +o(L)},
$$
as $L\to\infty$. Note that since $\eps>0$, we have $\mu<1$ for any $L\ge 1$. So, for any $\delta>0$ there exist $L$ large enough that 
\begin{equation}\label{eq:mean_progeny}
\mu\ge k^L \cdot e^{-L (\Lambda^\star(\gamma-\eps)+\delta)}.
\end{equation}
Fixing this value for $L$, and since $\mu<1$, (\ref{eq:chung_erdos}) yields 
$\pr{Z_i>0} \ge C \mu^i,$
for some fixed constant $C>0$ independent of $i$. More precisely, together with (\ref{eq:GW_embedding}), we obtain
\begin{align*}
\pr{M_{iL} \le (\gamma-\eps){iL}}
& \ge C k^{iL} e^{-iL (\Lambda^\star(\gamma-\eps)+\delta)},
\end{align*}
which proves the lower bound for all $m= iL$, for some $i\ge 0$. Finally, when $m=iL+r$ for $r\in \{1,\dots,L-1\}$ it suffices that $M_{iL}\le (\gamma-\eps)iL$ and $\sum_{j=1}^r Y_j \le (\gamma-\eps)r$, where $Y_j$, $j\ge 1$ are i.i.d.\ copies of $Y$. Since 
$$\inf_{1\le r<L} \pr{\sum_{i=1}^r Y_i \le (\gamma-\eps) r}>0,$$
the result follows easily for general integers $m$.
\end{proof}

\noindent\textsc{About the limiting constants for the dag model.} Before proceeding to the analysis of the circuit model, we use the tail bounds we have just devised to provide rough arguments that explain the values of the constants in our laws of large numbers. (The symbol $\approx$ is used in a very informal way.) The first main idea is that the dependence is small enough that the upper bound given by the union bound essentially yields the correct constant. By Theorem~\ref{thm:BRW-lefttail} we (should) have
\begin{align*}
\pr{D_n\ge c\log n}
& \approx \prc{M_{\floorc{c\log n}} \le \log n}\\
& \approx k^{c\log n} \cdot e^{-c\log n \Lambda^\star(1/c)}\\
& \approx n^{c(\log k - \Lambda^\star(1/x))}.
\end{align*}
In particular, we have $\lim_{n \to \infty} \pr{D_n\ge c\log n} = 0$ for every $c$ such that $\Lambda^\star(1/c)>\log k$. This suggests that  we might have $D_n \sim \lambda_k \log n$ in probability for $\lambda_k=\sup\{x: \Lambda^\star(1/c)<\log k\}$. The proof of this fact is the topic of Section~\ref{sec:typical_distance}. Similarly, using Theorem~\ref{thm:BRW-righttail} we obtain
\begin{align*}
\pr{\min_{x\ge n/2} D_x \le c \log n} 
& \le n \pr{D_{n/2}\le c \log n}\\
& \approx \pr{M_{\floorc{c\log n}} \ge \log (n/2)}\\
& \approx n e^{-k\alpha c \log n (1/c -\gamma)}\\
& \approx n^{1-k\alpha (1-c\gamma)},
\end{align*}
so that it should be the case that 
$$\frac{\min_{n/2 \leq x \leq n}D_x}{\log n} \to \left(1-\frac 1{k\alpha}\right)\lambda_k\quad \text{in probability}.$$ 
We prove this formally (when the limit constant above is non-negative; $\min_x D_x$ is non-negative) in Section~\ref{sec:minimum_distance}. In a similar way, we can see that Theorem~\ref{thm:BRW-lefttail} leads to a correct guess that $\max_{x \leq n} \frac{D_x}{\log n} \to ke$ in probability when the attachment distribution is uniform \cite{TX96}. 



\section{\bf Longest paths in random $k$-\textsc{dag}s: typical distance $D_n$}
\label{sec:typical_distance}

We start by studying the length $D_n$ of the longest path from node $n$ to the root of the tree. The typical distance $D_n$ is studied using methods similar to the study of the typical shortest path distance in \cite{DJ09}. Let $\Lambda^\star$ denote the rate function associated with the random variable $Y=-\log X$. Define
\begin{align}\label{eq:lambdak}
\lambda_k & =\sup\left\{z \geq 1/\ex{-\log X}:\; \Lambda^\star(1/z)\le\log k \right\}.
%
%
%
%
\end{align}
Note that we will consider attachment distributions with bounded density $f$ and thus $\ex{-\log X} = -\int_{0}^{1} (\log x) f(x) dx < \infty$. Moreover, as $X < 1$, $\ex{-\log X} > 0$.
%
%
%
%
\begin{theorem}
\label{thm:typlongest}
Let $k \geq 1$. Suppose $X\in(0,1)$ has a bounded density. Then, the longest path from node $n$ to the root in a $k$-\textsc{sarrd} with attachment $X$ satisfies, as $n\to\infty$,
\[
\frac{D_n}{\log n} \to \lambda_k
\]
in probability.
\end{theorem}
The theorem follows from Lemma \ref{lem:typupperb} and Lemma \ref{lem:typlowerb} below. 
From Theorem~\ref{thm:typlongest}, we easily obtain the asymptotics for the typical distance in the uniform circuit model discussed by \citet{AGM99} and \citet{TX96}. Indeed, if $X$ is uniform $(0,1)$, then the parents of a node $i$ are i.i.d.\ uniform in $\{0,1,\dots,i-1\}$. In this case, $-\log X\sim \text{Exponential}(1)$ and $\Lambda^\star(z)=z-1-\log z$. Then $\lambda_k$ is the only solution of the equation in $z$
\[
\Big(\frac{ke}{z}\Big)^z = e
\]
that is at least $1$. Numerical values are presented in Table~\ref{tab:num_uniform}. 

\begin{table}[htb]
\begin{tabular}{c | c c c c c}
$k$ & 2 & 3 & 4 & 5 & 10\\ \hline
$\lambda_k$ & 4.311070407\dots & 7.080786915\dots & 9.820440021\dots & 12.55049054\dots & 26.16346184\dots \\
\end{tabular}
\caption{\label{tab:num_uniform}Some numerical values for the constant $\lambda_k$ for the case of uniform recursive circuits.}
\end{table}


\subsection{\bf Upper bound on $D_n$}
\begin{lemma}\label{lem:typupperb}
For any $c>\lambda_k$, there exists $C$ and $\eta>0$ such that for all $n$ large enough, we have 
$$\pr{D_n \geq c \log n} \le C n^{-\eta}.$$
\end{lemma}
\begin{proof}
To simplify the notation, we prove  $\pr{D_n \geq c \log n + 2} \leq C n^{-\eta}$ for all $c > \lambda_k$, which is an equivalent statement. Let $R_n$ denote the number of hops to the root for a random path from $n$, i.e., a path that choses a uniformly random edge at every step. By the union bound, we get
\begin{equation}
\label{eq:typupperb_ub}
\pr{ D_n \geq c \log n + 2} \leq k^{c \log n + 2}\cdot \pr{R_n \geq c \log n + 2}.
\end{equation}
Let $X_1, X_2,\dots$ be i.i.d.\ copies of $X$. Then, to bound the right hand side, we let $t=\ceil{c\log n}$ and apply Markov's inequality
\begin{align}
\label{eq:typupperb_ch}
\pr{R_n \geq c \log n + 2} 
&\leq
\pr{R_n > t} \nonumber \\
&\leq \pr{ nX_1 \dots X_{t} \geq 1 } \notag \\
	&\leq \inf_{\lambda \geq 0} n^\lambda \exc{ X_1^\lambda \dots X_{t}^\lambda } \notag \\
	&= \inf_{\lambda \geq 0} n^{\lambda} \exc{X^\lambda}^t \notag \\
	&= \inf_{\lambda \geq 0} \exp{ \lambda\log n + \Lambda(-\lambda)t \Big.} \notag \\
 	&\leq \inf_{\lambda \geq 0} \exp{ \lambda\log n + \Lambda(-\lambda) c \log n  \Big.}, 
		&\text{(as $\Lambda(-\lambda) \leq 0$)} \notag \\
	&\leq \exp{ -\sup_{\lambda \geq 0} \left\{-\frac{\lambda}{c}  - \Lambda(-\lambda)\right\} c\log n} \notag \\ 
	&= \exp{ -c \Lambda^\star(1/c) \log n }.
\end{align}
Combining \eqref{eq:typupperb_ub} and \eqref{eq:typupperb_ch}, we get
\begin{align*}
\pr{D_n \geq c \log n + 2} 
&\leq k^2 \exp{ c \log n \left( \log k - \Lambda^\star(1/c) \right) }\\
&= k^2 n^{c (\log k-\Lambda^\star(1/c))}.
\end{align*}
As $\Lambda^{\star}$ is a decreasing function on $(-\infty, \ex{-\log X})$ and $c > \lambda_k$, we have $\Lambda^\star(1/c) > \log k$, which completes the proof.
\end{proof}

\subsection{\bf Lower bound on $D_n$}
The objective of this section is to prove the following lemma.
\begin{lemma}
\label{lem:typlowerb}
For any fixed $\eps > 0$, we have
\[
\pr{D_n < (1-\eps) \lambda_k \log n} \to 0,
\]
as $n\to\infty$. Furthermore, if there exists $\lambda>0$ such that $\exc{X^{-\lambda}}<\infty$, then
there exists constants $C$ and $\eta>0$ such that for all $n$ large enough
\begin{equation}\label{eq:Dn_tail}
\pr{D_n < (1-\eps) \lambda_k \log n} \le C n^{-\eta}.
\end{equation}
\end{lemma}

We start by describing the proof strategy and each subsection covers part of the proof.
Consider the set of ancestors of node $n$. These nodes form a directed acyclic graph where each node has out-degree $k$ except for the root. We start by considering the following model for the ancestor \textsc{dag}. Define a family of independent random variables $X'_{s}$ for all strings $s$ on $\cA=\{1,\dots, k\}$ that have the same distribution as $X$. For a string $s$ of length $\ell$, write $\ul s_0,\ul s_1,\ul s_2\dots, \ul s_\ell$ for the prefixes of $s$ of length $0,1,\dots, \ell$ respectively. The ancestor of node $n$ obtained by following the path $s$ is labeled
\[
L'(n, s) = \lfloor \dots \lfloor\lfloor n X'_{\ul s_1}\rfloor X'_{\ul s_2}\rfloor \cdots X'_{s} \rfloor.
\]
This defines a tree (indexed by the strings $s$) where distinct nodes can share the same label. In this ideal model, the length of the longest path that reaches a node labeled $0$ can be obtained with little effort from Theorem \ref{thm:minbrw1}.

In our \textsc{sarrd}, the labels actually correspond to nodes, and there is a unique node labeled $i$, for $i=0,\dots, n$.  The tree of ancestors of a node is then actually a \textsc{dag}, as the in-degree of a node can be more than one. In particular, this creates dependencies between the random variables $X$ along different paths even if they have no common prefix. 

In order to avoid having to deal with these dependencies, we use the following strategy to find a long path. Note that in our ideal setting, a path to the root corresponds to a path to a node labeled $0$. Starting at node $V_0 = n$, we look at all the ancestors of $n$ of order $\ell$ (i.e., $\ell$ jumps away from $n$) for some well-chosen large constant $\ell$. From all these possible paths, we pick the path (of length $\ell$) that reaches the node $V_1$ with the largest label. Then the same process is repeated for this node until a node with label $0$ is reached. This strategy defines a path, and we show that the length of this path can be made as large as $(\lambda_k - \eps) \log n$ with high probability. The advantage of using this method is that only a small portion of the tree is visited so it is easier to bound the ``collision'' probability. The path constructed for the ideal model can then be shown to be exactly the same in a real \textsc{sarrd} with high probability.

\subsubsection{\bf The ideal tree}
In this section, the objective is to obtain a good lower bound on the length of the longest path from node $n$ to a node of label $0$ in the ideal tree (Lemma~\ref{lem:Vq}). More precisely, let $V_j$ denote the \emph{label} obtained after $j$ steps (each one composed of $\ell$ jumps); the string defining the path from $n$ to the corresponding node  is denoted $S_j \in \cA^{j \ell}$. We have $V_0 = n, S_0 = \emptystring$ and for $j \geq 0$,
\begin{equation}
V_{j+1} = \max_{s \in \cA^{\ell}} L'(n, S_{j} \cdot s) \qquad \text{ and } \qquad S_{j+1} = S_{j} \cdot \argmax_{s \in \cA^{\ell}} L'(n, S_{j} \cdot s).
\end{equation}

The objective is now to show that starting from $n$ and after $q = \ceil{(1-\eps) \lambda_k \log (n)/\ell}$ steps of $\ell$ jumps, no node of label zero has been reached with probability going to $1$. If this happens, we clearly have a path of length at least $q\ell\ge(1-\eps)\lambda_k \log n$ between $n$ and a node labeled zero.

\begin{lemma}\label{lem:Vq}Suppose $\ex{-\log X} < \infty$. Let $\eps\in (0,1)$. For $q = \ceil{(1-\eps) \lambda_k \log (n)/\ell}$, we have
\[\prc{V_q\le n^{\eps/4}}\to 0,\]
as $n\to\infty$. Furthermore, if there exists $\lambda>0$ such that $\exc{X^{-\lambda}}<\infty$, then, there exists $\eta>0$ such that for all $n$ large enough, we have
\begin{equation}\label{eq:Vq_tail}
\prc{V_q \le n^{\eps/4}} \le n^{-\eta}.
\end{equation}
\end{lemma}

\begin{proof}
Recall that for a string $s$, we write $\underline s_0=\emptyset,\underline s_1, \underline s_2,\dots$ for the prefixes of $s$ of length $0,1,2,\dots$, respectively. Thus, for $j \geq 0$, we have by definition 
	\begin{align*}
	V_{j+1} &= \max\{ \lfloor \lfloor V_{j} X'_{S_{j} \cdot \underline s_1 }\rfloor \cdots X'_{S_{j} \cdot \underline s_\ell} \rfloor : s \in \cA^{\ell}\}\\
			&\geq V_{j} \cdot \max\{X'_{S_{j} \cdot \underline s_1}\times \dots \times X'_{S_j\cdot \underline s_\ell}: s\in \cA^\ell\}  - \ell.
	\end{align*}
Letting $Z_{j}^{(\ell)} = \max\{X'_{S_{j} \cdot \ul s_1} \times \cdots \times X'_{S_{j} \cdot s_\ell}:s\in \cA^\ell\}$, we have
	\begin{align*}
	V_{q} &\geq \left(\left( \left(n Z^{(\ell)}_{0} - \ell \right) Z^{(\ell)}_1 - \ell \right) \cdots Z^{(\ell)}_{q-1}\right) - \ell \geq n \prod_{j=0}^{q-1} Z^{(\ell)}_{j} - q \ell.
	\end{align*}
It follows immediately that
\begin{align*}
\prc{V_q \leq n^{\eps/4}} &\leq \pr{ n \prod_{j=0}^{q-1} Z^{(\ell)}_{j} - q \ell \leq n^{\eps/4}}\\
				&\le \pr{ \sum_{j=0}^{q-1} - \log Z^{(\ell)}_{j} \geq \left(1-\frac\eps3\right)\log n},
\end{align*}
for all $n$ large enough.
Observe that $-\log Z^{\ell}_0$ is nothing but the minimum at generation $\ell$ of a branching random walk with increments distributed as $-\log X$ . As $\ex{-\log X} < \infty$, using Theorem~\ref{thm:minbrw1}, we get
\begin{equation}
\label{eq:expminbrw}
\lim_{\ell \to \infty} \frac{\exc{-\log Z_0^{(\ell)}}}{\ell} = \frac{1}{\lambda_k}.
\end{equation}
We pick $\ell_0=\ell_0(\eps)$ so that for $\ell \geq \ell_0$, we have $(1-\frac \eps8)\lambda_k \exc{-\log Z_0^{(\ell)}} \leq \ell$. In the rest of the proof, $\ell = \ell_0$ is fixed and we let $Z$ be a random variable having the same distribution as $Z^{(\ell)}_0$.

We can now use a limit theorem for renewal processes (see for example \cite{GS01}, Chapter 10.2, Theorem 1, see also \cite{Gut2009}).
%
%
%
%
%
We have the convergence in probability:
\begin{equation}
\label{eq:renewal}
\frac{\max \{q : - \sum_{j=0}^{q-1} \log Z_j^{(\ell)} \leq (1-\frac \eps 3)\log n \}}{(1-\frac \eps 3)\log n} \inprob \frac{1}{\exc{- \log Z^{(\ell)}_0}}\ge \frac{ (1-\frac \eps 8)\lambda_k}{\ell},
\end{equation}
where the last inequality follows by our choice for $\ell$.
Thus, with probability going to one, and for $q = \ceil{(1-\eps) \lambda_k \log n/\ell}$ 
we have
\begin{equation}
\label{eq:probagen}
\lim_{n \to \infty} \prc{V_q \leq n^{\eps/4}} = 0. 
\end{equation}

We now move on to the proof of the tail bound (\ref{eq:Vq_tail}). Note that, in a stochastic sense, we have $Z_0^{(\ell)}\ge X_1 \cdot X_2 \cdots X_\ell,$
where $X_i$ are i.i.d.\ copies of $X$. Hence, since $\ell$ is fixed, there exists $\eta>0$ such that $\exc{Z^{-\eta}}<\infty$ so that, by Markov's inequality,
\begin{align*}
\pr{n\prod_{j=0}^{q-1}Z^{(\ell)}_{j} \le n^{\eps/3}}
&\le \left(\frac{\exc{e^{ -\eta\log Z }}}{e^{\eta \ell(1-\eps/3)/(1-\eps/2)\lambda_k}}\right)^{q},
\end{align*}
for all $n$ large enough.
To complete the proof, observe that, as $\eta\to0$, we have
\begin{align*}
\frac{\exc{e^{-\eta \log Z}}}{e^{\eta \ell(1-\eps/3)/(1-\eps/2)\lambda_k}}
&=1+\eta \left(\exc{-\log Z}-\frac{(1-\eps/3)\ell}{(1-\eps/2) \lambda_k}\right)+o(\eta)\\
&\leq1+\eta \frac{\ell}{\lambda_k}\left(\frac 1 {1-\frac \eps8}-\frac{1-\frac\eps3}{1-\frac \eps 2}\right) + o(\eta)<1
\end{align*}
for $\eta$ small enough by our choice for $\ell$; it is routine to verify that the second term is indeed negative for any $\eps\in (0,1)$. Since $q=\ceil{(1-\eps) \lambda_k \log (n)/\ell}$, the tail bound (\ref{eq:Vq_tail}) follows readily.
\end{proof}

We call the event studied in this section $A_n(q(n)):= \event{V_{q(n)} > n^{\eps/4}}$. To make the notation lighter, the dependence in $n$ will be omitted.

\subsubsection{\bf The real \textsc{dag}: handling collisions}

Recall that the statement in Lemma~\ref{lem:Vq} only deals with the ideal tree model. In order to prove that a long enough path also exists in the real \textsc{dag}, we couple the \textsc{sarrd} and the ideal tree in such a way that with high probability, the path $S_q$ that we exhibited in the ideal tree can also be found inside the \textsc{sarrd}. What we mean here is the following: the path $S_q$ defines a sequence of distinct \emph{labels} in the ideal tree; in the \textsc{sarrd} the \emph{nodes} corresponding to these labels are also along a path with high probability.

We use the random variables $X'_s$ that were used in the ideal tree, and new independent random variables $X''_{x, p}$ for $x \in \{0, \dots, n\}$ and $p\in \cA$. 
We start by defining an ordering of strings:
\begin{equation}
\label{eq:orderstring}
s \leq s' \quad\text{ if and only if }\quad |s| < |s'| \text{ or } |s| = |s'| \text{ and } s \leq_{\textrm{lex}} s'
\end{equation}
where $\leq_{\textrm{lex}}$ is the lexicographic order on the strings on the alphabet $\cA$. This is the \emph{breadth-first} order in the $k$-ary tree. Then, if one of the paths from $n$ has label $x$ in the ideal tree, let $T_x$ denote the first (or breadth-first) such path:
\[
T_x = \min \{ s : L'(n,s) = x, |s|\le n \}, 
\]
where the min is taken with respect to the order in \eqref{eq:orderstring} and we agree that $\min\{\emptyset\}=\infty$. The labels in the ideal tree are not distinct, and the corresponding vertices cannot directly represent nodes in the \textsc{dag}. We now ensure that a node $x$ in the \textsc{sarrd} corresponds to the \emph{first} vertex in the ideal tree with label $x$ (if such a vertex exists):
 for all $p \in \cA$ let
\[
X_{x,p} = \left\{ \begin{array}{ll}
X'_{T_x \cdot p} & \textrm{if $T_x \neq \infty $}\\
X''_{x,p} & \textrm{if $T_x = \infty $}.
\end{array} \right.
\]
Clearly, the random variables $X_{x,p}$ are independent and distributed as $X$, so that they define a proper $(X,k)$-\textsc{sarrd}.

We define the event that the sequence of labels along the path $S_q$ is the same in the \textsc{dag} generated by the variables $(X_{x,p}: 0\le x \le n, 1\le p \le k)$ and in the ideal tree generated by $(X'_{x,p}: 0\le x\le n, 1\le p \le k)$
\[
B(q) = \event{ L'(n, s) = L(n, s): s\preceq S_q}.
\]

In particular, if $B(q)$ and $A(q)$ both hold, then the path $S_q$ does not reach the root since its minimum label is at least $n^{\eps/4}$ and since $S_q$ is also a path in the \textsc{dag} we have $D_n\ge q\ell$. Write $E^c$ to denote the complement of an event $E$. Then, since $q=\ceil{(1-\eps) \lambda_k \log (n)/\ell}$, we have
\begin{align}\label{eq:bound_dnAB}
\pr{D_n < (1-\eps) \lambda_k \log n} 
&\le \pr{A(q)^c \cup B(q)^c} \nonumber\\
&= \pr{A(q)^c} + \pr{B(q)^c \cap A(q)}.
\end{align}
 Lemma~\ref{lem:Vq} bounds the probability that $A(q)$ does not hold so that to prove Lemma~\ref{lem:typlowerb}, it suffices to bound the probability that $B(q)^c\cap A(q)$ occurs.

\begin{lemma}
Let $\eps\in(0,1)$ and let $q=q(n) = \ceil{(1-\eps) \lambda_k \log(n)/\ell}$. Then, there exists $\eta>0$ such that for $n$ large enough
\[
\pr{A(q) \cap B(q)^c} \leq n^{-\eta}.
\]
\end{lemma}
\begin{proof}
In the following, we write $n_1=n^{\eps/4}$. Recall that $s^-$ denotes the prefix of $s$ of length $|s|-1$. If the path $S_q$ is a breadth-first path in the ideal tree (in the sense that each of its vertices is the first with its label) then our coupling ensures that $B(q)$ occurs. So for $B(q)$ to fail, there must be a node in $S_q$ that is not the first with its label. We have
\begin{align}
\label{eq:collisionbound}
\pr{B(q)^c \cap A(q)}
&\leq \pr{ \exists j< q, |s|\le \ell : T_{L'(V_j, s)} \neq s, T_{L'(V_j, s^-)} = s^-, V_q \geq n_1} \notag\\
&\leq \sum_{j=0}^{q-1} \pr{ \exists s : |s|\le \ell, T_{L'(V_j, s)} \neq s, T_{L'(V_j, s^-)} = s^-, V_q \geq n_1}.
\end{align}
For any fixed $j<q$, we have by the union bound
\begin{align}\label{eq:collision_term}
&\pr{ \exists s, |s|\le \ell : T_{L'(V_j, s)} \neq s, T_{L'(V_j, s^-)} = s^-, V_q \geq n_1}\notag \\
&\leq \sum_{|s|\le \ell} \pr{\exists s_0 < s : L'(V_j, s_0) = L'(V_j, s), T_{L'(V_j, s^-)} = s^-, V_q \geq n_1}\notag \\
&\leq \sum_{|s|\le \ell} \sum_{s_0 < s} \pr{L'(V_j, s_0) = \floor{X'_{s} L'(V_j, s^-)}, T_{L'(V_j, s^-)} = s^-, V_q \geq n_1}.
\end{align}
We now condition on $L'(V_j, s^-)$ and on $L'(V_j, s_0)$ and use the independence between the random variable $X_s$ and $(L'(V_j, s^-), L'(V_j, s_0))$:
\begin{align*}
&\pr{L'(V_j, s_0) = \floor{X'_{s} L'(V_j, s^-)}, T_{L'(V_j, s^-)} = s^-, V_q \geq n_1} \\
&\leq \sum_{u, v \le n} \pr{v = \floor{uX'_s}, L'(V_j, s^-) = u, L'(V_j, s_0) = v, L'(V_j, s^-) \geq n_1} \\
&\leq \sup_{u \geq n_1, v \le n} \pr{v = \floor{uX'_s}}.
\end{align*}
Since $X'_s$ has a density bounded by $b$, we have, for any $v\le n$ and $u\le n_1$,
\[
\pr{v = \floor{uX'_s}} = \pr{v \leq uX'_s < v+1} \leq \frac{b}{u} \leq \frac{b}{n_1}.
\]
Thus, going back to equations \eqref{eq:collisionbound} and \eqref{eq:collision_term}, since there are at most $q k^{2(\ell+1)}$ terms, each one at most $b/n_1$, we obtain
\begin{align*}
\pr{B(q)^c \cap A(q)} 
				&\leq \frac{b q k^{2(\ell+1)}}{n_1}, 
\end{align*}
which readily yields the result since $n_1=n^{\eps/4}$, $q = \ceil{(1-\eps) \lambda_k \log(n)/\ell}$ and $b, k, \ell$ are fixed constants.
\end{proof}
Thus, recalling Lemma~\ref{lem:Vq} and (\ref{eq:bound_dnAB}), we have
\[
\lim_{n \to \infty} \pr{D_n < (1-\eps) \lambda_k \log n} = 0.
\]
In the case when the attachment distribution $X$ is such that $\exc{X^{-\lambda}}<\infty$ for some $\lambda>0$, there exists $\eta>0$ such that the following stronger bound holds for all $n$ large enough
\[
\pr{D_n < (1-\eps) \lambda_k \log n} \le n^{-\eta},
\]
which finishes the proof of Lemma \ref{lem:typlowerb}.


\section{\bf Longest paths in random $k$-\textsc{dag}s: minimum distance $\min_{n/2 \leq x \leq n} D_x$}
\label{sec:minimum_distance}

In this section, we suppose that the attachment distribution $X$ has a bounded density and satisfies
\begin{equation}\label{eq:X_tail}
\qquad \pr{X\le t}=t^{\alpha+o(1)}, 
\end{equation}
as $t\to0$, for some $\alpha\in(0,\infty)$. Note that this implies that $\exc{X^{-\alpha/2}} = \int_{0}^{\infty} \prc{X^{-\alpha/2} \geq t} dt < \infty$. The lower tails of the step size should clearly influence the distances; the present setting with an underlying branching structure points towards the dependence in \eqref{eq:X_tail}. The value of $\alpha$ determines the value of the minimum distance for a random recursive $(k, X)$-\textsc{dag}.  Define the constant
\begin{equation}\label{eq:def_beta}
\beta:= \max \left(1-\frac 1{k\alpha}, 0\right).
\end{equation}

\begin{theorem}
\label{thm:minlongest}
Let $k \geq 2$ and $X$ be as above. The minimum longest path distance in a random recursive $(k,X)$-\textsc{dag} satisfies
\[
\frac{\min_{n/2 \leq x \leq n} D_x}{\log n} \to \beta \lambda_k,
\]
in probability,
where $\lambda_k = \sup\left\{z:\; \Lambda^\star(1/z) \le \log k \right\}$ is the constant defined in \eqref{eq:lambdak}.
\end{theorem}

%
%
In the case where $\pr{X \leq t} \geq t^{\alpha + o(1)}$ for any $\alpha > 0$, we obtain a limit of $0$. In fact, the upper bound in Lemma \ref{lem:minlongestupperb} only uses a lower bound on $\pr{X \leq t}$ for $t$ in the neighborhood of $0$. Similarly, if $\pr{X \leq t} \leq t^{\alpha + o(1)}$ for any $\alpha > 0$, the limit becomes $\lambda_k$ as the lower bound of Lemma \ref{lem:minlongestlowerb} only uses an upper bound on $\pr{X \leq t}$ for $t$ in the neighborhood of $0$. 
%
%

For $X\sim \text{uniform}[0,1)$, we have $\alpha=1$ and $\beta=1 - 1/k$. This yields $\min_{n/2 \leq x \leq n} D_x= (k-1)/k\cdot \lambda_k \log n)$ for uniform random recursive \textsc{dag}s. In particular, for $k=2$, we obtain
\[
\frac{D_n}{\log n} \to\lambda_2 \qquad \text{ and } \qquad \frac{\min_{n/2 \leq x \leq n} D_x}{\log n} \to \frac{\lambda_2}{2}
\]
where $\lambda_2 =  4.31107\dots$. This completes the table of asymptotic properties of different natural distances in uniform random \textsc{dag}s \cite[Table 1]{DJ09}.

The reader will easily be convinced when reading the proof that the result remains unchanged if one considers $\min\{D_i: \delta n\le i\le n\}$, for some $\delta>0$. We keep the current statement for simplicity.

\subsection{\bf Upper bound on $\min_{n/2 \leq x \leq n} D_x$}

\begin{lemma}
\label{lem:minlongestupperb}
For any $\eps > 0$, there exists $\eta>0$ such that, for all $n$ large enough,
\[
\pr{\min_{n/2 \leq x \leq n} D_x \geq (\beta+\eps) \lambda_k\log n}\le n^{-\eta}.
\]
\end{lemma}
\begin{proof}
The strategy is to show that there exists a node $x \in \{\ceil{n/2}, \dots, n\}$ such that all its parents have label at most $n^{\beta+ \eps}$. We then conclude using Theorem \ref{thm:typlongest}.
For $n/2 \leq x \leq n$, we look at the $k$ parents $L(x, 1), \dots, L(x, k)$ of $x$ in the \textsc{dag}. Since $\beta\ge 1-\frac 1{k\alpha}$ and $x\le n$, we have
\begin{align*}
\pr{\max\{L(x,1), \dots, L(x,k)\} \leq n^{\beta+\eps}} 
									&= \pr{ x X \leq n^{\beta+\eps}}^k \\
									&\geq \prm{ X \leq n^{-1/(k\alpha)+\eps}}^k \\
									&\geq n^{-1+k\alpha\eps/2}
\end{align*}
for all $n$ large enough, as $X$ satisfies \eqref{eq:X_tail}. Define
$$
N = \big|\big\{x : \ceil{n/2} \le x\le n, \max_{1\le i\le k} L(x,i) \leq n^{\beta+\eps} \big\}\big|.$$ We have 
\begin{equation}\label{eq:exp_N}
\ex{N} \geq \left(n - \ceil{\frac{n}{2}} + 1\right) \cdot n^{-1+k\alpha \eps/2} \geq \frac{n^{k\alpha \eps/2}}{2}.
\end{equation}
Observe that the events $\{\max(L(x,1), \dots, L(x,1)) \leq n^{1-1/(k\alpha)+\eps}\}$ are independent for different nodes $x$. So we can also compute the variance of $N$:
\[
\var N \leq \sum_{x=\ceil{n/2}}^n \pr{\max\{L(x,1), \dots, L(x,k)\} \leq n^{\beta+\eps}} = \ex{N}.
\]
Thus, using the second moment inequality, sometimes called Chung--Erd\H{o}s inequality \citep{CE52}, and since $\exc{N}\to\infty$ by \eqref{eq:exp_N}, we obtain 
\begin{equation}
\label{eq:existslongjump}
\pr{N =0} \leq \frac{\var{N}}{\var N + \ex{N}^2} \le \frac{1}{1+\exc{N}}=O(n^{-k\alpha \eps/2}).
\end{equation}

Now, define the random node
\[
V = \max \left\{ \{0\} \cup \left\{ x: \ceil{n/2}\le x\le n, \max(L(x,1), \dots, L(x,k)) \leq n^{\beta+\eps}\right\} \right\}.
%
%
%
\]
It only remains to show that such a node $V$ has a small longest path to the root. It is sufficient to bound the following probability, with $\eps<1$. We have
\begin{align*}
&\pr{\min_{n/2 \leq x \leq n} D_x \geq  (\beta+3\eps) \lambda_k \log n + 1} \\
&\leq \pr{D_{V} \geq (\beta + 3\eps) \lambda_k \log n + 1, N>0} + \pr{N=0} \\
&\leq \sum_{s=1}^k \pr{D_{L(V,s)} \geq (\beta + 3\eps) \lambda_k \log n, N > 0 } + \pr{N = 0} \\
&\leq \sum_{s=1}^{k}\pr{D_{L(V,s)} \geq (\beta + 3\eps) \lambda_k \log n, L(V,s) \leq n^{\beta+\eps}} + \pr{N=0} \\
&\leq k \pr{D_{\lfloor n^{\beta+ \eps}\rfloor} \geq (1+\eps)(\beta+\eps) \lambda_k \log n} + 
\pr{N=0} \qquad\text{for }\eps<1.
\end{align*}
The last inequality holds because we can condition on $L(V,s)$ and use the independence of the events $\event{L(V,s) = v}$ and $\{D_v \geq (1 + \eps) (\beta+\eps) \lambda_k \log n\}$ for any $v$.
We can then conclude using Lemma~\ref{lem:typupperb} and \eqref{eq:existslongjump}.
\end{proof}

\subsection{\bf Lower bound on $\min_{n/2 \leq x \leq n} D_x$}
The objective of this section is to prove the following lower bound. Note that when $\beta=0$, no lower bound is needed to prove Theorem~\ref{thm:minlongest}, so that we can safely assume here that $\beta>0$. In particular, we have $\beta=1-\frac 1{k\alpha}$.
\begin{lemma}\label{lem:minlongestlowerb}
Suppose that $\beta>0$. For any $\eps > 0$, there exists $\eta > 0$ such that
\[
\pr{\min_{n/2 \leq x \leq n} D_x \leq (\beta-\eps) \lambda_k \log n} = O(n^{-\eta}).
\]
\end{lemma}
Clearly, it is sufficient to consider $\eps<\beta$; we do so until the end of the proof. By the union bound, we have
\begin{align*}
\pr{\min_{n/2 \leq x \leq n} D_x \leq (\beta-\eps) \lambda_k \log n} 
&\le \sum_{x = \ceil{n/2}}^{n}  \pr{D_x \leq (\beta-\eps) \lambda_k \log n}\\
&\le n \cdot \sup_{n/2\le x\le n} \pr{D_x \le (\beta-\eps) \lambda_k \log n}.
\end{align*}
So it suffices to show that $\prc{D_x \leq (\beta-\eps) \lambda_k \log n} = o(1/n)$ uniformly for all $x \in \{\ceil{n/2}, \dots, n\}$. In the rest of the proof we fix $x \in \{\ceil{n/2}, \dots, n\}$. 

If $D_x \leq (\beta-\eps) \lambda_k \log n$, this means that all the ancestors $a$ of $x$ of order $h$ (think of $h = \floor{10 \log \log n}$) have depth $D_a \leq (\beta-\eps) \lambda_k \log n - h$. To bound the probability of such an event, we show that there are many distinct ancestors of order $h$ that have labels at least $n^{\beta- \delta}$ (Proposition \ref{prop:probex} that is based on Lemma \ref{lem:manylarge} and \ref{lem:collision}). Then using the explicit bound in Lemma \ref{lem:typlowerb}, we prove that the probability that none of these ancestors have direct parents that have typical depth is small.

More precisely, we build a small \textsc{dag} of ancestors of $x$ up to $h = \floor{10 \log \log n}$ generations. We start by showing that many of the ancestors of order $h$ have labels at least $n^{\beta - \delta}$.
As a warm up, we first bound the probability that a single ancestor $h$ levels away from $x$ has a low label.
\begin{lemma}
\label{lem:problongjump}
Let $x\ge n/2$ be fixed. Let $s \in [k]^h$ where $h = \floor{10 \log \log n}$ and $u \geq h$. For any $\eta > 0$ and $n$ large enough,
\[
\prc{L(x,s) \leq u} \leq \left( \frac{u}{n} \right)^{\alpha - \eta}
\]
In particular, for $\delta \in (0, \beta)$ and $u = n^{\beta - \delta}$, we get
\begin{equation}
\label{eq:special-u}
\prc{L(x,s) \leq n^{\beta - \delta} } \leq n^{-1/k-\alpha\delta/2}.
\end{equation}
\end{lemma}
\begin{proof}
We have the following bound
\begin{align*}
\prm{L(x,s) \leq u} 
&\leq \prm{ nX_1 \dots X_{h} - h \leq u } \\
&\leq \prm{ X_1 \dots X_{h} \leq \frac{2u}{n} }\\
& = \prm{ (X_1 \dots X_{h})^{\gamma} > (2u/n)^{\gamma}},
\end{align*}
for all $n$ large enough and any $\gamma<0$. By Markov's inequality and the independence of $X_1,\dots, X_h$, it follows that 
\begin{align*}
\prm{L(x,s) \leq n^{\beta-\delta}} 
	&\leq (2u/n)^{-\gamma} \cdot \ex{ X_1^\gamma \dots X_h^\gamma }\\
	&=  (2u/n)^{-\gamma} \cdot \ex{X^\gamma}^h.
\end{align*}
Now observe that, for any $\eps>0$, we have 
$$\exc{X^{-\alpha+\eps}}=\int_0^{\infty} \prc{X^{-\alpha+\eps}>t}dt<\infty,$$ by our assumption on the tail of $X$ in \eqref{eq:X_tail}. Choosing $\gamma=-\alpha+\eps$ for small enough $\eps$, we obtain
\[
\prm{L(x,s) \leq u}
	\le (2u/n)^{\alpha- \eta/2} \leq (u/n)^{\alpha - \eta}
\]
for $n$ large enough. We obtain \eqref{eq:special-u} by choosing $\eta$ small enough.
\end{proof}

In order to handle the dependence between different paths of the \textsc{dag}, we bound the number of path intersections far from the root. More precisely, order the strings $s$ according to the order defined in \eqref{eq:orderstring} (breadth-first order). 
Define the set of nodes that are ancestors of $x$ along paths indexed by the words $s'<s$: $\cV(x,s) = \{ L(x,s'), s' < s \}$. Then we say that a path labeled by $s$ \emph{collides} if 
\[
L(x,s) \in \cV(x,s) \qquad \text{ and } \qquad L(x,\ul s_t) \notin \mathcal V(x, \ul s_{t}), \quad1 \leq t < |s|,
\]
where $\ul s_i$ is the prefix of $s$ of length $i$.
Of course, the chance to collide is greater if the labels are small; for us it will be sufficient to consider the nodes with label at least $n^{\beta-\delta}$.
Define the number of paths of length at most $h$ colliding at nodes with label at least $n^{\beta-\delta}$:
$$N_c = \big| \{ s: |s|\le h,  L(x,s) \geq n^{\beta-\delta} \text{ and } s \text{ collides} \} \big|.$$

\begin{lemma}
\label{lem:collision}
Let $\delta > 0$. For all $n\ge 1$ and all $i\ge 0$, we have
\[\pr{N_c \geq i} \leq k^{2i(h+1)} b^i n^{-i(\beta-\delta)}.
\]
\end{lemma}
\begin{proof}Let $i\ge 0$, by definition, we have
\begin{align}\label{eq:Nc}
\pr{N_c\ge i}
& = \pr{\exists s_1<\dots< s_i: s_1,\dots, s_i~\text{collide}}\nonumber\\
& \le k^{i(h+1)}\sup_{s_1<\dots<s_i}\!\!\!\!\pr{s_1,\dots, s_i~\text{collide}}.
\end{align}
We prove by induction that for all $i\ge 0$
\begin{equation}\label{eq:Nc_induction}
\sup_{s_1<\dots<s_i} \!\!\!\!\pr{s_1,\dots, s_i~\text{collide}}\le \left(b k^{h+1} n^{\delta-\beta}\right)^i.
\end{equation}
Clearly, (\ref{eq:Nc}) and (\ref{eq:Nc_induction}) together imply the result.
The base case, $i=0$, is clear. Suppose now that $i\ge 1$. Let $s_1<s_2<\dots<s_i$. Write $S_{i}=\{s_1,\dots, s_i\}$ and $S_{i-1}=S_i\setminus \{s_i\}$. Note that $S_{i-1}$ is empty if $i=1$. We say that a set $S$ collides if all its elements collide. Recall that $\ol s$ denotes the last symbol of the word $s$. Then
\begin{align*}
&\pr{s_1,\dots, s_i~\text{collide}}\\
&=\pr{S_i~\text{collides}}\\
&=\prm{S_{i-1}~\text{collides}, L(x,s_i) \geq n^{1-1/k-\delta}, L(x,s_i^-) \notin \cV(x,s_i^-), L(x,s_i) \in \cV(x,s_i) } \\
&= \sum_{u,W} \pr{S_{i-1}~\text{collides}, L(x,s_i^-) = u, \cV(x,s_i) = W, u \notin \cV(x,s_i^-), \floor{u X_{u, \ol s_i}} \in W },
\end{align*}
where the sum in the last line ranges on $u\ge n^{1-1/(k\alpha)-\delta}$ and $W\subseteq \{0,\dots, n\}$. In the rest of the proof, $W$ will always be a subset of $\{0,\dots, n\}$ and we do not always remind it to keep the equations as light as possible.

We claim that for any fixed $W\subseteq \{0,\dots, n\}$ and any $u$, the event 
$$\event{S_{i-1}~\text{collides}, L(x,s_i^-) = u, \cV(x,s_i) = W, u \notin \cV(x,s_i^-)}\quad \text{and} \quad \event{\floor{u X_{u, \ol s_i}} \in W}$$ are independent. The latter event is clearly determined by $X_{u,p}$, $1\le p\le k$. 
Consider now the other event. 
In order to determine whether it occurs or not, it suffices to look at the ancestors $L(x,s')$ of $x$ for $s' < s_i$. If the value of $X_{u,\ol s_i}$ is needed to compute one of these ancestors, then the event does not hold because we would then have $u \in \cV(x,s^-)$. Otherwise, we can determine whether the event holds or not without looking at $X_{u,\ol s_i}$.

For $|s|\le h$, we have $|\cV(x,s)|\le k^{h+1}$ and we can then write
\begin{align*}
\label{eq:onecollision}
\pr{S_i~\text{collides}}
&\le \pr{S_{i-1}~\text{collides}}\sup_{u \geq n^{\beta - \delta}, |W| \leq k^{h+1} } \!\!\!\!\!\!\pr{ \floor{u X } \in W}\\
&\leq \pr{S_{i-1}~\text{collides}}\cdot k^{h + 1} b n^{-\beta+\delta}\\
&\le \big(k^{h + 1}b n^{-\beta+\delta}\big)^i,
\end{align*}
by the induction hypothesis. This completes the proof.
\end{proof}

Recall that our aim is to prove that, with probability at least $1-o(1/n)$, $x$ has many distinct anscestors of order $h$, all of them should have a large enough label. To make this precise, define the event
$$E_x = \big\{x \text{ has at least } k^{h-3} \text{ distinct ancestors of order } h \text{ with labels at least } n^{\beta- \delta}\big\}.$$
We want to prove that $\pr{E_x}=1-o(1/n)$. 
We first show the following decomposition for $E_x$. We have already proved that $\pr{N_c\ge 3}=o(1/n)$, and it will then suffice to bound the probability that the first event in \eqref{eq:Ex} below does not occur.
\begin{lemma}For every $x\ge n/2$ we have
\begin{equation}\label{eq:Ex}
\big\{  \exists p \in [k],  \forall s_p \in [k]^h : s_p[1] = p \text{ and } L(x,s_p) > n^{\beta - \delta}\big\} \cap \big\{N_c \leq 2\big\} \subseteq E_x.
\end{equation}
\end{lemma}
\begin{proof}
The first event in \eqref{eq:Ex} ensures that all the ancestor of $L(x,p)$ of order $h-1$ have labels greater than $n^{\beta- \delta}$ for some $p \in [k]$. These $k^{h-1}$ nodes need not be distinct. Consider the paths in the order defined by \eqref{eq:orderstring}. By definition, if a path does not collide and has no prefix that collides, then its label is distinct from all the previous ones. It follows that only the paths that do have a prefix counted by $N_c$ might not have distinct labels. Here $N_c\le 2$ and it is simple to see that a pair of paths whose collision maximizes the number of potential duplicates are $p\cdot 2$ and $p\cdot3$ if $k \geq 3$ and $p\cdot 2$ and $p\cdot 12$ if $k=2$. In any case, of the $k^{h-1}$, there are at least $k^{h-3}$ nodes with distinct labels.
\end{proof}

To complete the proof that $\pr{E_x}=1-o(1/n)$, it only suffices to bound the probability of the first event in \eqref{eq:Ex} not occurring.
\begin{lemma}
\label{lem:manylarge}
Let $\delta \in (0,\beta/4)$. For $n$ large enough,
\[
\prm{ \exists s^1, \dots, s^k \in [k]^{h-1} : \forall p \in [k], L(x,p\cdot s^p) \leq n^{\beta - \delta} } \leq n^{-1-\delta/4}.
\]
\end{lemma}
\begin{proof}
Let $A_i$ be the event that $L(x, s^j) \leq n^{\beta - \delta}$ for $j\le i$. We prove by induction on $i \geq 1$, for $n$ large enough
\begin{equation}
\label{eq:manylarge}
\prc{A_i} \leq n^{-i/k-\delta/8}.
\end{equation}
The base case, for $i=1$, follows from Lemma~\ref{lem:problongjump}.
Suppose now that $i \geq 2$. The difficulty to prove the induction step relies in the dependence between the events $A_{i-1}$ and $L(x, s^i)\le n^{\beta-\delta}$. We introduce the following notation for the paths from node $x$
\[
P(x,s) = \{L(x, \ul s_1), L(x,\ul s_2),\dots, L(x, s) \} \qquad \text{ and } \qquad P_{i-1} = \bigcup_{1\le j<i} P(x, s^j).
\]
Note that $x$ is not included in the paths.
To upper bound the left-hand side in \eqref{eq:manylarge}, we condition on the first time that the path $P(x,s_i)$ reaches the set $P_{i-1}$:
\begin{align}
\prc{A_i}
=\prc{A_i, P(x,s^i) \cap P_{i-1} = \emptyset }+ \sum_{1\le t\le h} \prm{A_i, L(x, {\ul s^i_{t-1}}) \notin P_{i-1},  L(x, \ul s^i_t) \in P_{i-1}}.\label{eq:manylarge-term1} 
\end{align}
In the following, $W$ and $Q$ will always denote a subset of $\{0,\dots, n\}$; we do not always remind it.

\medskip
\noindent \emph{i.\ The path $P(x,s^i)$ does not collide. }The first term in \eqref{eq:manylarge-term1}, on the event $P(x,s^i)\cap P_{i-1}\neq \emptyset$, is the easiest to deal with:
\begin{align*}
&\prc{A_i , P(x,s^i) \cap P_{i-1} = \emptyset } = \sum_{W,Q:W \cap Q = \emptyset} \!\!\!\!\!\!\!\!\prc{A_i, P(x,s^i) = Q, P_{i-1} = W}.
\end{align*}
As $Q \cap W = \emptyset$, the events 
$$\big\{L(x, s^i) \leq n^{\beta- \delta}, P(x,s^i) = Q\big\}\quad \text{and}\quad\big\{L(x, s^j) \leq n^{\beta - \delta} \text{ for }j< i, P_{i-1} = W\big\}$$ 
are independent. In fact, the first event is in the sigma-algebra generated by $\{X_{v,p} : v \in Q, p \in [k]\}$ and the second in the one generated by $\{X_{v,p} : v \in W, p \in [k]\}$. Thus, we obtain
\begin{align}
\prc{A_i, P(x,s^i) \cap P_{i-1} = \emptyset }
&\leq \prc{A_{i-1}} \cdot \prm{L(x,s^i) \leq n^{\beta- \delta}} \notag \\
&\leq \prc{A_{i-1}} \cdot n^{-1/k - \alpha\delta/2}, \label{eq:manylarge-ineq1}
\end{align}
where the last inequality follows from Lemma~\ref{lem:problongjump}.


\medskip
\noindent\emph{ii.\ The path $P(x,s^i)$ collides. }We next look at the terms \eqref{eq:manylarge-term1} that correspond to cases when there are some collisions, i.e., $P(x,s^i)\cap P_{i-1}\neq \emptyset$. Recall that $t$ is the location of the \emph{first} collision on $P(x, s^i)$. In the following, we write $a^i_t$ for the $t$-th symbol of $s^i$ (so $s^i=a^1_1a^i_2\dots a^i_{|s^i|}$). For $t \in \{1, \dots, h\}$, we have
\begin{align*}
&\prm{A_i, L(x, {\ul s^i_{t-1}}) \notin P_{i-1},  L(x, \ul s^i_t) \in P_{i-1}}\\
&= \sum_{|W| \leq kh} \sum_{u \not\in W}\prm{A_i, P_{i-1} = W, L(x, {\ul s^i_{t-1}}) = u, \lfloor u X_{u, a^i_t}\rfloor \in W}
\end{align*}
We separate this sum into two terms depending on whether $u \leq n^{\beta-\delta}$ or $u > n^{\beta-\delta}$. The sum over $u \leq n^{\beta-\delta}$ can be bounded as in \eqref{eq:manylarge-ineq1}: since $s_i|_t$ is the first path that hits $P_{i-1}$, we have $P(x,\ul s^i_{t-1})\cap P_{i-1}=\emptyset$. It follows that
\begin{align}
&\sum_{|W| \leq kh} \sum_{u \leq n^{\beta-\delta}, u \notin W} \prm{A_i, P_{i-1} = W, L(x, {\ul s^i_{t-1}}) = u, \lfloor u X_{u, a^i_t}\rfloor \in W}\notag\\
&\le \pr{A_{i-1}} \cdot \prm{L(x,\ul s^i_t)\le n^{\beta-\delta}}\notag\\
&\le \pr{A_{i-1}}\cdot n^{-1/k - \alpha\delta/2},\label{eq:manylarge-ineq2} 
\end{align}
by Lemma~\ref{lem:problongjump}.
We now look at the sum over $u > n^{\beta-\delta}$.
\begin{align}
&\sum_{ |W| \leq kh}\sum_{u > n^{\beta-\delta}, u \notin W} \prm{A_i, P_{i-1} = W, L(x, {\ul s^i_{t-1}}) = u, \floorc{u X_{u, a^i_t}} \in W} \notag\\ 
&\leq \sum_{u > n^{\beta-\delta}}\sum_{|W|\le kh, W \not\ni u} \sum_{Q\cap W=\emptyset} \prm{A_{i-1}, P_{i-1}=W, L(x,\ul s^i_{t-1})=u, P(x,\ul s^i_{t-1})=Q, \floorc{u X_{u, a^i_t}} \in W} \notag\\
&= \sum_{u > n^{\beta-\delta}}  \sum_{|W|\le kh, W \not\ni u} \sum_{Q\cap W=\emptyset} \prc{A_{i-1}, P_{i-1}=W} \cdot \prm{L(x,\ul s^i_{t-1})=u, P(x,\ul s^i_{t-1})=Q} \cdot \prc{\floorc{uX} \in W} \notag\\
&\leq \sum_{u > n^{\beta-\delta}} \left( \sup_{|W| \leq kh} \prc{\floor{uX} \in W} \right) \cdot \prm{L(x,\ul s^i_{t-1})=u} \cdot \pr{A_{i-1}}. \label{eq:manylarge-ineq3}
\end{align}
In the first equality, we used the independence of the events $\{A_{i-1}, P_{i-1} = W\}, \{L(x, \ul s^i_{t-1}) = u, P(x,\ul s^i_{t-1})=Q\}$ and $\{X_{u, a^i_t} \in W\}$ if $u \notin W$ and $Q \cap W = \emptyset$. Using the fact that $X$ has a density bounded by $b$, we have for any $W$ of size at most $kh$,
\[
\pr{\floor{uX} \in W} \leq \frac{bkh}{u}.
\]
The next step is to bound the sum in equation \eqref{eq:manylarge-ineq3} by considering groups of nodes in the intervals $\left(2^{\ell} \floorc{n^{\beta-\delta}}, 2^{\ell+1} \floorc{n^{\beta-\delta}}\right]$ for non-negative integers $\ell$. We have for any $\ell$ integer and any $\eta \in (0,\alpha)$
\begin{align*}
&\sum^{2^{\ell+1} \floorc{n^{\beta-\delta}}}_{u = 2^{\ell}\floorc{n^{\beta - \delta}} + 1} \left( \sup_{|W| \leq kh} \prc{\floor{uX} \in W} \right) \cdot \prm{L(x,\ul s^i_{t-1})=u} \\
&\qquad \qquad \leq \pr{L(x,\ul s^i_{t-1}) \leq 2^{\ell+1} n^{\beta-\delta}} \cdot \frac{bkh}{2^{\ell}n^{\beta - \delta}} \\
&\qquad \qquad \leq  \left(\frac{2^{\ell+1} n^{\beta-\delta}}{n}\right)^{\alpha - \eta} \frac{bkh}{2^{\ell}n^{\beta - \delta}} \\
&\qquad \qquad \leq  2^{\alpha} b kh \cdot  (2^{\ell}n^{\beta - \delta})^{\alpha - \eta - 1} n^{-\alpha + \eta} 
\end{align*}
for $n$ large enough. For the second inequality, we used Lemma~\ref{lem:problongjump}.
Now if $\alpha \leq 1$, recalling that $\beta = 1 - 1/(k\alpha)$, we can bound 
\begin{align*}
n^{-\alpha+\eta} \cdot  (2^{\ell}n^{\beta - \delta})^{\alpha - \eta - 1} &\leq n^{-\alpha+\eta} 
		(n^{\beta - \delta})^{\alpha - 1} \\
		&\leq n^{-\alpha+\eta}  n^{\alpha - 1/k - \alpha \delta - \beta + \delta} \\
		&\leq n^{-1/k - \delta}
\end{align*}
for small enough $\eta$ and $n$ large enough. For the case $\alpha > 1$, we get
\[
n^{-\alpha+\eta} \cdot  (2^{\ell}n^{\beta - \delta})^{\alpha - \eta - 1} \leq n^{-\alpha+\eta} \cdot  (2^{\ell}n^{\beta - \delta})^{\alpha - 1} \leq 2^{\alpha - 1} n^{-1}.
\]
provided $2^{\ell}n^{\beta - \delta} \leq 2n$. This shows that total weight of an interval can be bounded by
\[
\sup_{\ell} \sum^{2^{\ell+1} \lfloor n^{\beta-\delta} \rfloor}_{u = 2^{\ell}\lfloor n^{\beta - \delta}\rfloor + 1} \left( \sup_{|W| \leq kh} \prc{\floor{uX} \in W} \right) \cdot \prm{L(x,s^i_{t-1})=u}
\leq n^{-1/k-\delta/3}
\] 
for $n$ large enough, where the supremum is taken over non-negative integers $\ell$ such that $2^{\ell} n^{\beta-\delta} \leq n$. This allows us to bound the expression in \eqref{eq:manylarge-ineq3} by $\pr{A_{i-1}} n^{-1/k-\delta/2}$ as there are at most $\log_2 n$ intervals. By summing this term with the term corresponding to the sum for $u \leq n^{\beta-\delta}$, we obtain
\begin{equation}
\label{eq:manylarge-ineq4}
\prm{A_i, L(x, \ul s^i_{t-1}) \notin P_{i-1},  L(x, \ul s^i_t) \in P_{i-1}} \leq \pr{A_{i-1}} n^{-1/k- \min(\alpha, 1) \delta}
\end{equation}
for any $t \in \{1, \dots, h\}$. Now it only remains to bound the sum \eqref{eq:manylarge-term1} over the different values of $t$ using \eqref{eq:manylarge-ineq1} and \eqref{eq:manylarge-ineq4}. Using the induction hypothesis to bound $\pr{A_{i-1}}$ we obtain the desired induction step \eqref{eq:manylarge}.
\end{proof}

Putting these results together, we obtain
\begin{proposition}
\label{prop:probex}
Let $x \in \{\lfloor n/2\rfloor ,\dots, n\}$. Let $\delta > 0$. For $n$ large enough,
\[
\prc{E_x} \geq 1 - n^{-1-\delta/4} + n^{-3/2+4\delta}.
\] 
\end{proposition}

We are now in position to prove the lower bound claimed in Lemma~\ref{lem:minlongestlowerb}. Fix $\eps, \delta \in (0,1/10)$ and $h = \floor{10 \log \log n}$.
First, using the lower bound on the typical distance, we can say that most of the nodes in the tree have $D_y \geq (1-\eps) \lambda_k \log n$. We want to show that with high probability, \emph{for every }$y$ such that $n/2\le y\le n$ we have $D_y\ge (\beta-\eta)\lambda_k \log n$ for arbitrarily small $\eta$. Define the set of ``bad nodes'' that violate the property
$$\cB = \big\{y: y\ge n/2, D_y < (1-\eps)(\beta-2\delta)\lambda_k \log n\big\},$$
that we decompose in a dyadic fashion: for any positive integer $r$, define
\[
\cB_r = \big\{y: 2^r \leq y < 2^{r+1},  D_y < (1-\eps)(\beta-2\delta) \lambda_k \log n\big\}.
\]
Using Lemma \ref{lem:typlowerb} and Markov's inequality, we have for every $r$ such that $(\beta-2\delta)\log_2 n \leq r \leq \log_2 n$
\[
\pr{\left| \cB_r \right| \geq \frac{2^{r}}{100 \cdot b}} \le 100 b \cdot 2^{-r} \exc{|\cB_r|} \le C n^{-\eta},
\]
for some $C$ and $\eta>0$ independent of $r$. Recall that $b$ is a bound on the density of $X$.
As a result the event 
\begin{equation}\label{eq:eventA}
A = \event{ |\cB_r| < 2^{r}/(100b) \text{ for }(\beta-2\delta)\log_2 n\leq r \leq \log_2 n }
\end{equation}
is such that $\pr{A}=1 - O(n^{-\eta/2})$ for all $n$ large enough. 

Thus,
\begin{align}
&\pr{\min_{n/2 \leq x \leq n} D_x < (1-\eps) (\beta - 2\delta) \lambda_k \log n + h + 1} \notag \\
&\leq \pr{A^c} + \sum_{x = \ceil{n/2}}^n \pr{A, D_x < (1-\eps)(\beta - 2\delta) \lambda_k \log n + h + 1} \notag \\
&\le \pr{A^c} + \sum_{x = \ceil{n/2}}^n \pr{A, \forall s \in [k]^h: D_{L(x,s)}< (1-\eps)(\beta - 2\delta) \lambda_k \log n + 1 }. \label{eq:boundminlongest}
\end{align}

We now bound the term $\prc{A, \forall s \in [k]^h: D_{L(x,s)}< (1-\eps)(\beta - 2\delta) \lambda_k \log n + 1 }$ by conditioning on the event $E_x$. Let $S_1, \dots, S_{k^{h-3}}$ denote, when $E_x$ holds, a set of paths that lead (when starting at $x$) to distinct nodes.
\begin{align*}
&\prc{A, \forall s \in \cA^h: D_{L(x,s)}< (1-\eps)(\beta - 2\delta) \lambda_k \log n + 1} \\
&\leq \pr{E_x^c} + \pr{A, E_x, \forall s \in \cA^h: D_{L(x,s)}< (1-\eps)\left(\beta- 2\delta \right) \lambda_k \log n+1} \\
&\leq \pr{E_x^c} + \pr{A, E_x, \forall i \in \{1, \dots, k^{h-3}\}: D_{L(x,S_i)} < (1-\eps)\left(\beta - 2\delta \right) \lambda_k \log n + 1} \\
&\leq \pr{E_x^c} + \pr{A, E_x, \forall i \in \{1, \dots, k^{h-3}\}, p \in \cA: \floor{L(x,S_i) X_{L(x,S_i),p}} \in \cB}
\end{align*}
By taking the worst possible set $\cB$ (compatible with the event $A$) we obtain the bound:
\begin{align}\label{eq:supbadnode}
&\nonumber\pr{A, E_x, \forall i \in \{1, \dots, k^{h-3}\}, p \in \cA: \floor{L(x,S_i) X_{L(x,S_i),p}} \in \cB} \\
&\leq  \left(\sup_{B, u \geq n^{\beta-\delta}}\!\!\!\!\!\! \pr{ \floor{u X} \in B } \right)^{k \cdot k^{h-3}}
\end{align}
where the maximization is taken over all sets $B$ such that $|B \cap [2^r, 2^{r+1})| < 2^{r}/(100b)$ for all $r$. Note that a key point here is that the nodes $L(x, S_i)$ for $i \in \{1, \dots, k^{h-3}\}$ are distinct and all have labels at least $n^{\beta-\delta}$. It now suffices to bound the right-hand side of \eqref{eq:supbadnode}.
\begin{lemma}
\label{lem:problandbadnode}
Suppose $B \subseteq \{0, \dots, n\}$ such that for all $r$ satisfying $(\beta-2\delta)\log_2 n\leq r \leq \log_2 n$, we have
\[
|B \cap [2^r, 2^{r+1})| < 2^r/(100b).
\]
Then for any $y \geq n^{\beta - 
\delta}$, we have for $n$ large enough
\[
\pr{\floor{yX} \in B} \leq 1/2.
\]
\end{lemma}
\begin{proof}
Let $r_{\max}$ the largest integer such that $2^{r_{\max}} \leq y$ and let $r_{\min}$ be the smallest integer at least as large as $(\beta-2\delta)\log_2 n$. We then have
\begin{align*}
\pr{\floor{yX} \in B} 	&\leq \pr{\floor{yX} \leq 2^{r_{\min}}} + \sum_{r=r_{\min}}^{r_{\max}} \pr{\floor{yX} \in B \cap [2^r, 2^{r+1})} \\
				&\leq \pr{yX \leq 2n^{\beta-2\delta} + 1 }+ \frac{b}{y} \cdot \frac{1}{100 b} \left( 2^{r_{\max}+1} + 2^{r_{\max}} + \dots + 2^{r_{\min}} \right)  \\
				&\leq 3 b n^{-\delta}+ \frac{2^{r_{\max} + 2}}{100 y}  \\
				&\leq 1/2
\end{align*}
for $n$ large enough.
\end{proof}
Getting back to equation \eqref{eq:boundminlongest}, we get
\begin{align*}
&\pr{\min_{n/2 \leq x \leq n} D_x < (1-\eps) \left(\beta- 2\delta \right) \lambda_k \log n + h+1}\\
&\leq \pr{A^c} + \sum_{x = \ceil{n/2}}^n \left( \pr{E_x^c} + (1/2)^{k^{h-2}} \right) \\
&= O(n^{-\eta})
\end{align*}
for small enough $\eta$ using Proposition~\ref{prop:probex}.
This concludes the proof of Lemma~\ref{lem:minlongestlowerb} and Theorem~\ref{thm:minlongest}.

\section{\bf Concluding remarks}
We studied the longest path distance in a general class of random recursive \textsc{dag}s. The parents of node $x$ are chosen independently and distributed as $\floor{xX}$ for some random variable $X \in [0,1)$. When $X$ has a bounded density, we proved laws of large numbers for the typical and minimum distance. 
For both of these results, the upper bounds do not make any assumption on the attachment distribution $X$. We use the condition of bounded density for the lower bound when bounding the dependencies between the different paths up to the root. It would be interesting to extend these results to more general distributions. More generally, under which conditions is it possible to translate a result in an ideal model like the branching random walk considered in Section~\ref{sec:BRW} into a result about a real model that exhibits limited dependencies?

\section*{\bf Acknowledgments}
We would like to thank Louigi Addario-Berry, Luc Devroye and Nicolas Fraiman for many helpful discussions as well as the anonymous referee for useful comments. Part of this work was completed while the authors were visiting McGill's Bellairs Research Institute, Barbados. NB's research is supported by the ANR-09-BLAN-0011  project Boole.

\bibliographystyle{abbrvnat}
\bibliography{longest}

\end{document}